\documentclass[oneside,reqno,UKenglish,autoref]{amsart}
\usepackage[a4paper, total={6in, 8in}]{geometry}
\usepackage{amssymb,amsfonts,amsthm,thm-restate}

\usepackage{url}
\usepackage{hyperref}

\usepackage[capitalise,noabbrev,sort]{cleveref}

\newcommand{\EPAGRP}{\ensuremath{\textsc{EPAG-N}}\xspace}
\newcommand{\EPAG}{\ensuremath{\textsc{EPAG}}\xspace}
\newcommand{\EPA}{\ensuremath{\textsc{EPA}}\xspace}
\newcommand{\EPAD}{\ensuremath{\textsc{EPAD}}\xspace}
\newcommand{\IPGCD}{\ensuremath{\textsc{IP-GCD}}\xspace}

 \usepackage{bm}
 \usepackage{xspace}
 \usepackage{mathtools}
 \usepackage{tikz-cd}
 \mathtoolsset{showonlyrefs}
 
 \newcommand{\bi}{\begin{itemize}}
\newcommand{\ei}{\end{itemize}}
\newcommand{\be}{\begin{enumerate}}
\newcommand{\ee}{\end{enumerate}}

  \renewcommand{\ge}{\geqslant} \renewcommand{\le}{\leqslant}

\newcommand{\NP}{\ensuremath{\mathsf{NP}}\xspace} %
\newcommand{\NEXP}{\ensuremath{\mathsf{NEXP}}\xspace} %

\newcommand\GIVENBY{\text{given by}}
\newcommand\AND{\text{and}}
\newcommand\OR{\text{or}}
\newcommand\WHERE{\text{where}}
\newcommand\FOR{\text{for}}
\newcommand\FORALL{\text{for all}}
\newcommand\FORSOME{\text{for some}}

\renewcommand{\implies}{\Rightarrow}
\renewcommand{\iff}{\Leftrightarrow}

\newcommand{\N}{\mathbb{N}}
\newcommand{\Z}{\mathbb{Z}}
\newcommand{\B}{\mathcal{B}}
\newcommand{\X}{\mathcal{X}}
\newcommand{\R}{\mathcal{R}}
\newcommand{\cP}{\mathcal{P}}
\newcommand{\I}{\mathcal{I}}
\newcommand{\J}{\mathcal{J}}
\newcommand{\M}{\mathcal{M}}

\newcommand{\swfi}[1]{{#1}^{\pm1}} 
\newcommand{\pres}[2]{\left\langle{#1}\mid{#2}\right\rangle}
\DeclareMathOperator{\Inv}{Inv}
\DeclareMathOperator{\im}{im}
\newcommand{\Epi}{\ensuremath{\textsc{Epi($\mathcal{I}$,$\B$)}}\xspace}

\newcommand{\ol}{\overline}
\newcommand{\wh}{\widehat}
\newcommand{\wt}{\widetilde}
\newcommand{\id}{\mathrm{id}}
\renewcommand{\mod}{\;\operatorname{mod}\;}

\newtheorem{decision}{Decision Problem}

\theoremstyle{plain}
\newtheorem{theorem}{Theorem}[section]
\newtheorem{lemma}[theorem]{Lemma}
\newtheorem{proposition}[theorem]{Proposition}

\theoremstyle{definition}

\newtheorem{theoremx}{Theorem}

\title[Testing epimorphism onto the bicyclic monoid is \NP-complete]
{Testing epimorphism onto the bicyclic monoid is \NP-complete}

\author[L. Carroll]{Luka Carroll}
\address{School of Mathematical and Physical Sciences, University of Technology Sydney, Ultimo NSW 2007, Australia}
\email{Luka.M.Carroll@student.uts.edu.au}

\author[M. Elder]{Murray Elder}
\address{School of Mathematical and Physical Sciences, University of Technology Sydney, Ultimo NSW 2007, Australia}
\email{murray.elder@uts.edu.au}

\date{\today}



\begin{document}

\begin{abstract}
We prove that deciding whether there is a surjective homomorphism from an arbitrary finitely presented inverse monoid onto the bicyclic monoid is \NP-complete. As part of the proof, 
we  show that an extension of  existential Presburger arithmetic which involves greatest common divisors on $n$ arguments
is in \NP, 
extending a recent result of D\'efossez, Haase, Mansutti, and  P\'erez (SODA 2024).

\end{abstract}

\maketitle

\section{Introduction}

The aim of this paper is to prove the following.
\begin{theoremx}\label{thmMAINA:epi_is_NP-complete}
    The following decision problem is \NP-complete: on input a finite 
    inverse monoid presentation for an inverse monoid $\I$, decide whether there exists a surjective homomorphism from  
    $\I$ to the bicyclic monoid.
\end{theoremx}
We show the problem is in \NP by converting it into a 
logical formula to be satisfied.
Specifically, we reduce it to a particular extension of Presburger arithmetic. 
Let \EPAG denote the existential fragment of the first-order theory of the structure
\begin{equation}
    \langle\Z;0,1,+,\le,(\gcd(\cdot,\cdot)\le d)_{d\in\N_+},(\gcd(\cdot,\cdot)= d)_{d\in\N_+}\rangle.
\end{equation}
Let \EPAGRP denote the extension of \EPAG consisting of all true sentences of the form
\begin{equation}
    \exists\bm{x}\in\Z^n:\Phi(\bm{x})\land\bigwedge_{i=1}^k\gcd(f_{i,1}(\bm{x}), \dots, f_{i,q_i}(\bm{x}))\sim_i d_i
\end{equation}
where $\Phi$ is an open \EPAG formula, 
$f_{i,j}$ are linear polynomials with integer coefficients,
 ${\sim_i}\in\{\le,=\}$, and  $d_i\in\N_+$.

\begin{theoremx}\label{thm:MAINB_EPAGRP_poly_sol}
    If a sentence $\exists\bm{x}\in\Z^n:\Phi(\bm{x})$ is in $\EPAGRP$ then there exists a solution $\bm{a}\in\Z^n$ to $\Phi$ with polynomial bit length. Consequently, deciding membership in \EPAGRP is \NP-complete. 
\end{theoremx}
To prove Theorem~\ref{thm:MAINB_EPAGRP_poly_sol} we first prove a number-theoretic result, which is itself a consequence of a result of Iwaniec \cite{iwaniec1978jacobsthal}
on the maximal length of intervals 
containing multiples of an arbitrary set of primes.

\begin{theoremx}\label{thm:MAINC_k_aryGCD_to_binaryGCD}
    There exists $c>0$ such that for all $x_1,\dotsc,x_k\in\Z$ with $x_1\ne0$ there exists $0\le m_3,\dotsc,m_k\le c\log^2|x_1|$ such that 
    \[\gcd(x_1,x_2+m_3x_3+\cdots+m_kx_k)=\gcd(x_1,\dotsc,x_k).\]
\end{theoremx}

See \cref{sec:InvMon} for the relevant details on finitely presented inverse monoids and the bicyclic monoid. Deciding whether there is an epimorphism from a finitely presented \emph{group}
to a fixed group or class of finitely presented groups \cite{elderVAbel,FriedlL,grochow_et_al,Kup-zombies,Kup-colouring,Remes}, as well as between graphs \cite{subgraph_epi},  has been studied by several authors. The significance of the present work is that it adds to this collection and generalises from groups to inverse monoids. It also demonstrates the difficulty of what might appear to be an easy first case, the bicyclic monoid.

Although
it is introduced here 
as a tool to study epimorphism of inverse monoids, Theorem~\ref{thm:MAINB_EPAGRP_poly_sol} appears to be a new result in its own right, building on work of \cite{IntegerProgGCD,lechner2015complexity} as discussed in \cref{sec:PAstuff}.

After submission to arXiv,  OpenAI announced a new bound for  Jacobsthal's function \cite{OAI:A-quadratic-bound-for-Jacobsthals-function-September-25-2026}.  Replacing \autoref{thm:C(r)_bound} by the sharper (for simplicity quadratic) bound would lead to a bound of $c\left(\frac{\log|x_1|}{\log\log|x_1|}\right)^2$ in \autoref{thm:MAINC_k_aryGCD_to_binaryGCD}. However we chose not to incorporate this.

The paper is organised as follows. In \cref{sec:CoprimeTechnical,sec:PAstuff} we prove Theorems~\ref{thm:MAINB_EPAGRP_poly_sol} and~\ref{thm:MAINC_k_aryGCD_to_binaryGCD} respectively. We then give a short introduction to inverse monoids and the bicyclic monoid in \cref{sec:InvMon}. We establish \NP-hardness for epimorphism onto the bicyclic monoid in \cref{sec:NPhard} via a polynomial time reduction from 3-SAT, and in \cref{sec:inNP} we prove the problem is in \NP.

\subsubsection*{Statement on use of AI tools}

Claude \cite{claude2026lessinfo}
was used in the development of \cref{sec:CoprimeTechnical}. Specifically
we entered the following query: given coprime integers $x_1,\dotsc,x_n$, are there any known results that could bound the size of coefficients needed to write a linear combination of $x_2,\dotsc,x_n$ that is coprime with $x_1$? Claude sketched a proof using the functions $C$ and $\omega$  
(defined in \cref{sec:CoprimeTechnical}) referencing \cite{iwaniec1978jacobsthal,heyman2015gcdshiftedsets}, which provided the foundation of 
the
proof of an initial version of \cref{prop:ternaryGCD_to_binaryGCD} with the hypothesis that $\gcd(x,y,z)=1$.
Using those ideas we we able to arrive at the current version of this proposition which gives 
\cref{thm:MAINC_k_aryGCD_to_binaryGCD}. 
The remainder of the results outside \cref{sec:CoprimeTechnical} were obtained independently of any AI tools, and 
none 
were used to assist in the writing or typesetting of the paper.

\subsubsection*{Notation}
For $n\in\N$ we write $[n]$ for the set $\{1,\dots, n\}$.
We denote the bit length (up to a constant factor) of 
an encoding of an object $x$ by~$\|x\|$. For an integer $x\in\Z$ we declare this to be its bit length in binary: $\|x\|=\lceil\log|x|\rceil+1$. The encoded size $\|\Phi\|$ of a logical formula $\Phi$ is the number of symbols it contains plus the total encoded size of any integers it contains. If $\bm{a}=(a_1,\dotsc,a_n)\in\Z^n$ then $\|\bm{a}\|=n\cdot\max\{\|a_i\|:i\in[n]\}$. We write $poly$ to represent an arbitrary but fixed polynomial in a single variable. Note that 
\begin{equation}
	poly(n)+poly(n) \ = \ poly(n)\cdot poly(n) \ = \ poly(poly(n)) \ = \ poly(n).
\end{equation} 
For $A\subseteq \Z$ and $f,g:A\to \mathbb{R}$ we write $f(x)\ll g(x)$ to mean there exists a constant $c\in\N_+$ independent of $x$ such that $f(x)\le cg(x)$ for all $x\in A$. 

Recall that for $x_1,\dotsc,x_k\in\Z$, $k\ge2$,
their greatest common divisor $\gcd(x_1,\dotsc,x_k)$ is the unique integer $d\in\N_+$ satisfying 
\begin{equation}
    d \mid x_i \ \ \text{for all $i\in[k]$} \qquad\AND\qquad c\mid x_i \ \ \text{for all $i\in[k]$} \ \ \implies \ \ c\mid d.
\end{equation}
when at least one $x_i\ne0$. Otherwise $\gcd(x_1,\dotsc,x_k)=0$.
Bezout's identity states that   
given $x_1,\dots, x_k\in \Z$, there exist $m_1,\dotsc,m_k\in\Z$ such that $m_1x_1+\cdots+m_kx_k=\gcd(x_1,\dotsc,x_k)$.  

\section{From $k$-ary  to binary GCD with logarithmic coefficients}
\label{sec:CoprimeTechnical}
In this section we prove 
\cref{thm:MAINC_k_aryGCD_to_binaryGCD}.
First we recall some basic facts from number theory.
We write $x\equiv_n y$ for $x,y,n\in\Z$ with $n\ge1$ to indicate $x$ is congruent to $y$ modulo $n$, and  $x\mid y$ for  $x$ divides $y$. 
 An integer $x\in\Z$ has a unique multiplicative inverse in $\Z_n=\{0,1,\dotsc,n-1\}$ if and only if $\gcd(x,n)=1$. If $n$ is prime then this holds for every $x\in\Z$ such that $n\nmid x$. The Chinese Remainder Theorem (CRT) states that, if $n_1,\dotsc,n_k\in\Z^+$ is such that $\gcd(n_i,n_j)=1$ for all $i\ne j$, and $ N=n_1\cdots n_k$, then the mapping
\begin{equation}\label{eqn:CRT}
	\Z_N \rightarrow \Z_{n_1}\times\cdots\times\Z_{n_k}\colon x\mapsto(x\mod n_1,\dotsc,x\mod n_k) 
\end{equation}
is a bijection. In particular, if $(x_1,\dotsc,x_k)\in \Z_{n_1}\times\cdots\times\Z_{n_k}$ then there exists $y\in\Z_N$ such that $y\equiv_{n_i}x_i$ for all $1\le i\le k$.

Let $\mathbb{P}_n=\{P\subseteq\N:(\forall p\in P)\; \text{$p$ is prime},\,|P|=n\}$ for $n\in\N$. Define
\begin{equation}
    C(n)=\max\{m\in\N:(\forall P\in\mathbb{P}_n,\,\exists c\in\N,\,\forall x\in[c,c+m-1],\,\exists p\in P)\; p\mid x\}
\end{equation}
where $[c,c+m-1]=\{c,c+1,\dotsc,c+m-1\}$. That is, varying over all sets $P$ of $n$ distinct primes, $C(n)$ is the size of the longest interval $I\subseteq\N$ where each $x\in I$ is divisible by some $p\in P$. Clearly $C(n)$ is monotonically increasing on $\N$.

A bound for $C(n)$ is due to Iwaniec \cite{iwaniec1978jacobsthal}.
\begin{lemma}[{\cite{iwaniec1978jacobsthal}}]\label{thm:C(r)_bound}
    For all $n>1$, $C(n)\ll (n\log n)^2$.
\end{lemma}
Define $\omega(x)$ to be the number of distinct prime factors of $x\in\Z\setminus\{0\}$. Using \cref{thm:C(r)_bound} and the well-known bound $\omega(x)\ll\frac{\log x}{\log\log x}$ 
it can been shown \cite{heyman2015gcdshiftedsets} that
\begin{lemma}[{\cite{heyman2015gcdshiftedsets}}]\label{lem:logSquared}
For all $x\in\Z$ with $|x|>1 $, 
$	C\big(\omega(x)\big) \ll \log^2 |x| $.
\end{lemma}

\begin{proposition}
\label{prop:ternaryGCD_to_binaryGCD}
  Let $x,y,z\in\Z$ with $x\ne 0$. There exists $m\in\N,m\ll\log^2|x|$ such that
    \[\gcd(x,y+mz)=\gcd(x,y,z).\]
\end{proposition}

\begin{proof}
    If $x=\pm1$ then the result holds with $m=0$. Suppose instead that $|x|>1$ and let $p_1^{r_1}\cdots p_k^{r_k}$ be the prime factorisation of $x$. Because $\gcd(x,y,z)\mid x$ we have 
    \begin{equation}
        \gcd(x,y,z)=p_1^{s_1}\cdots p_k^{s_k} \quad\FORSOME\quad s_1,\dotsc,s_k\in\N \quad\text{where each}\quad s_i\le r_i.
    \end{equation}
    Similarly, because $\gcd(x,y,z)\mid \gcd(x,y+mz)\mid x$ for all $m\in\N$ we have
    \[
        \gcd(x,y+mz)=p_1^{t_{m,1}}\cdots p_k^{t_{m,k}} \quad\FORSOME\quad t_{m,1},\dotsc,t_{m,k}\in\N \quad\text{where each}\quad s_i\le t_{m,i}\le r_i.
    \] 
    Hence it suffices to show that there exists $m\ll\log^2|x|$ such that for all $i\in[k]$ we have $t_{m,i}=s_i$ (or, equivalently, $p_i^{s_i+1}\nmid x$ or $p_i^{s_i+1}\nmid y+mz$). There are three cases for $i\in[k]$.\medskip
    \begin{description}
        \item[Case 1] If $s_i=r_i$ then $p_i^{s_i+1}\nmid x$ for all $m\in\N$.\smallskip
        \item[Case 2] Suppose $s_i<r_i$ and $p_i^{s_i+1}\mid z$. The former implies $p_i^{s_i+1}\mid x$, which combined with the latter implies $p_i^{s_i+1}\nmid y$ since $p_i^{s_i+1}$ is not a factor of $\gcd(x,y,z)$. Therefore $y+mz\equiv_{p_i^{s_i+1}}y\not\equiv_{p_i^{s_i+1}}0$ and thus $p_i^{s_i+1}\nmid y+mz$ for all $m\in\N$ in this case.\smallskip
        \item[Case 3] Suppose $s_i<r_i$ and $p_i^{s_i+1}\nmid z$. Then $y=a_ip_i^{s_i}$ and $z=b_ip_i^{s_i}$ for some $a_i,b_i\in\Z$ with $p_i\nmid b_i$. This means $b_i$ has a multiplicative inverse $b_i^{-1}$ modulo $p_i$ and thus
    \begin{equation}
    \begin{array}{rclll}
         m \not\equiv_{p_i} -b_i^{-1}a_i  \quad\implies\quad a_i+mb_i \not\equiv_{p_i} 0 & \implies & a_i+mb_i \ne qp_i & & \text{for all $q\in\Z$} \\
         & \implies & a_ip_i^{s_i}+mb_ip_i^{s_i} \ne qp_i^{s_i+1} & & \text{for all $q\in\Z$} \\
         & \implies & y+mz \ne qp_i^{s_i+1} & & \text{for all $q\in\Z$} \\
         & \implies & p_i^{s_i+1}\nmid y+mz.
    \end{array}
    \end{equation}
    \end{description}
    \noindent In the first two cases we have $p_i^{s_i+1}\nmid x$ or $p_i^{s_i+1}\nmid y+mz$ for all $m\in\N$, so we need only find $m\ll\log^2|x|$ such that $m \not\equiv_{p_i} -b_i^{-1}a_i$ for all $i\in I:=\{i\in[k]:s_i<r_i \ \AND \  p_i^{s_i+1}\nmid z\}$ as per the third case. By the CRT there exists $n\in\N$ such that $-n\equiv_{p_i}-b_i^{-1}a_i$ for all $i\in I$. Choose
    \begin{equation}
        m=\min\{m\in\N:p_i\nmid n+m \text{ for all $i\in I$}\}
    \end{equation}
    and note that $m\le C(|I|)\le C(\omega(x))\ll\log^2|x|$ by \cref{lem:logSquared}. Then
    \[
        n+m\not\equiv_{p_i} 0 \quad\implies\quad m\not\equiv_{p_i}-n\equiv_{p_i}-b_i^{-1}a_i \qquad\FORALL\ i\in I.\qedhere
    \]
\end{proof}

\begin{proof}
[Proof of \cref{thm:MAINC_k_aryGCD_to_binaryGCD}]
    For all $x_1,\dotsc,x_k\in\Z$ we have
    \begin{equation}
        \gcd(x_1,\dotsc,x_k)=\gcd\!\big(\!\gcd(x_1,\dotsc,x_{k-1}),x_k).
    \end{equation}
    Then the result follows from \cref{prop:ternaryGCD_to_binaryGCD} using induction on $k$.
\end{proof}

\section{Small-witness property of \EPAGRP formulae}
\label{sec:PAstuff}

We now turn to the proof of \cref{thm:MAINB_EPAGRP_poly_sol}. We start by defining the types of logical formulae we wish to consider. For this section, let $\bm{x}=(x_1,\dotsc,x_n)$ be a vector of variables, let $f,g,h$ denote arbitrary linear polynomials with integer coeffecients and indeterminates $\bm{x}$, and let $\sim$ denote $=$ or $\le$. \emph{Presburger arithmetic} (PA) is the first-order theory of the structure $\langle \Z;0,1,+,\le\rangle$. This is the set of all true first-order sentences\footnote{\emph{i.e.} expressions of the form $Q_1 x_1\dots Q_nx_n\Phi(\bm{x})$ where each $Q_i\in\{\forall,\exists\}$ and $\Phi$ is a formula, where a formula is defined inductively as: an atom, $\neg \Phi_1$, $\Phi_1\wedge \Phi_2$, or $\Phi_1\vee\Phi_2$ where $\Phi_1, \Phi_2$ are formulae.} with atoms of the form $f(\bm{x})\sim g(\bm{x})$ and variables ranging over $\Z$. The \emph{existential fragment of Presburger arithmetic} (\EPA) is the subset of PA containing only existentially quantified sentences. We will be considering three extensions of \EPA:

\be
\item \emph{\EPA with binary divisibility predicates} (\EPAD) is the set of all true, existentially quantified first-order sentences with atoms of the form $f(\bm{x})\sim g(\bm{x})$ or $f(\bm{x})\mid g(\bm{x})$.
\item \emph{\EPA with binary GCD predicates} (\EPAG) is the set of all true, existentially quantified first-order sentences with atoms of the form $f(\bm{x})\sim g(\bm{x})$ or $\gcd\!\big(f(\bm{x}),g(\bm{x})\big)\sim d$ where $d\in\N_+$. Note that the latter is indeed a binary predicate: $d$ is not a variable. Rather, there exists a distinct GCD predicate for each $d\in\N_+$ and ${\sim}\in\{\le,=\}$.
\item \emph{\EPAG with n-ary GCD constraints} (\EPAGRP) is the set of all true sentences of the form
\begin{equation}
    \exists\bm{x}\in\Z^n:\Phi(\bm{x}) \land \bigwedge_{i=1}^k \gcd\!\big(f_{i,1}(\bm{x}),\dotsc,f_{i,p_i}(\bm{x})\big)\sim_i d_i
\end{equation}
where $\Phi(\bm{x})$ is a quantifier-free \EPAG formula with free variables $\bm{x}$, each $d_i\in\N_+$ is a constant, and each ${\sim_i}\in\{\le,=\}$. In contrast to the binary GCD predicates in \EPAG, each constraint $\gcd\!\big(f_{1}(\bm{x}),\dotsc,f_{q}(\bm{x})\big)\sim d$ can have any number of arguments.
\ee

If $\exists\bm{x}\in\Z^n:\Phi(\bm{x})$ is a sentence in any of the aforementioned extensions of \EPA then we say that $\Phi$ is \emph{feasible}. This is equivalent to there existing a \emph{solution} $\bm{a}\in\Z^n$ to $\Phi$, meaning $\bm{a}$ is such that $\Phi(\bm{a})$ evaluates to true.

The problem of deciding membership in \EPAD is known to be in \NEXP due to the small-witness property of \EPAD formulae~\cite{lechner2015complexity}, and has recently been shown to be \NEXP-complete~\cite{barros2026nexpcompletenessexponentialcoefficientgrowth}. In order to obtain a polynomial bound on the size of solutions to \EPAGRP formulae,  we first appeal to the following theorem from \cite{lechner2015complexity} to obtain an exponential bound.

\begin{theorem}
[{\cite[Theorem~14]{lechner2015complexity}}]\label{thm:EPAD_NEXP}
If $\exists\bm{x}\in\Z^n:\Phi(\bm{x})$ is a sentence in $\EPAD$ then there exists a solution $\bm{a}\in\Z^n$ to $\Phi$ such that $\|\bm{a}\|\le 2^{poly(\|\Phi\|)}$.
\end{theorem}

\begin{lemma}\label{lem:EPAG-C_NEXP}
    If $\exists\bm{x}\in\Z^n:\Phi(\bm{x})$ is a sentence in $\EPAGRP$ then there exists a solution $\bm{a}\in\Z^n$ to $\Phi$ such that $\|\bm{a}\|\le 2^{poly(\|\Phi\|)}$.
\end{lemma}
\begin{proof}
    We prove the claim by showing that there exists a sentence $\exists\bm{x}\in\Z^n,\exists\bm{y}\in\Z^m:\Psi(\bm{x},\bm{y})$ in $\EPAD$ such that $\bm{a}\in\Z^n$ is a solution to $\Phi$ if and only if $(\bm{a},\bm{b})\in\Z^{n}\times\Z^m$ is a solution to~$\Psi$ for some $\bm{b}\in\Z^m$, and $\|\Psi\|=poly(\|\Phi\|)$.

    First, let $\Phi'$ be the logically equivalent formula obtained from $\Phi$ by using De Morgan's law to push all negations inwards until they apply directly to atoms. Note that by definition no negations apply to atoms of the form $\gcd(a_1,\dots,a_p)\sim d$, $p>2$, in $\Phi'$.

    Next we show that each literal of the form
    \begin{equation}\label{eqn:EPAG-C_atoms}
        \neg(\gcd(a,b)\sim d) \qquad\AND\qquad \gcd(a_1,\dotsc,a_p)\sim d \quad\WHERE\quad p>2
    \end{equation}
    can be replaced by \EPAD formulae. For this, we note the following equivalences:
    \begin{equation}\label{eqn:EPAG-C_EPAD_equiv1}
    \begin{array}{rcl}
        \neg(\gcd(a,b)= d) & \iff &  \exists c\in\Z: \neg(c=d)\land\gcd(a,b)=c, \\
        \neg(\gcd(a,b)\le d) & \iff &  \exists c\in\Z: \neg(c\le d)\land\gcd(a,b)=c, \\
        \gcd(a_1,\dotsc,a_p)\le d & \iff &  \exists c\in\Z: c\le d\land\gcd(a_1,\dotsc,a_p)=c.
    \end{array}
    \end{equation}
    Next, recall that if $c$ is a common divisor of $a_1,\dotsc,a_p\in\Z$ then $c$ divides any linear combination of $a_1,\dotsc,a_p$. Hence if $d\in\N$ is such that $d\mid a_i$ for all $i\in[p]$ and if there exist $m_1,\dotsc,m_p\in\Z$ for which $m_1a_1+\cdots+m_pa_p=d$ then $d=\gcd(a_1,\dotsc,a_p)$. Bezout's identity gives the converse. That is,
    \begin{equation}\label{eqn:EPAG-C_EPAD_equiv2}
    \begin{array}{rcl}
         \gcd(a_1,\dotsc,a_p)=d & \iff & \exists u_1,\dotsc,u_p\in\Z:\bigwedge_{i=1}^p d\mid a_i \land \bigwedge_{i=1}^p a_i\mid u_i \land \sum_{i=1}^p u_i=d. 
    \end{array}
    \end{equation}
    By applying to $\Phi'$ the equivalences in \eqref{eqn:EPAG-C_EPAD_equiv1}, then those in \eqref{eqn:EPAG-C_EPAD_equiv2} (introducing fresh variables in each case) we obtain a logically equivalent formula $\Phi''$ whose atoms are all \EPAD atoms. We then pull out the newly introduced quantifiers (which remain existential since negations apply only to atoms in $\Phi''$) to obtain the sentence $\exists\bm{x}\in\Z^n,\exists\bm{y}\in\Z^m:\Psi(\bm{x},\bm{y})$ where $\bm{y}$ is a tuple containing the newly introduced variables, which is in $\EPAD$ since $\Phi$ is feasible by assumption. Clearly the sentence $\exists\bm{x}\in\Z^n,\exists\bm{y}\in\Z^m:\Psi(\bm{x},\bm{y})$ fulfills its requirements.

    By \cref{thm:EPAD_NEXP}, there exists a solution $(\bm{a},\bm{b})$ to $\Psi$ with $\|\bm{a}\|,\|\bm{b}\|\le2^{poly(\|\Psi\|)}=2^{poly(\|\Phi\|)}$. Therefore $\bm{a}$ is a solution to $\Phi$ with $\|\bm{a}\|\le2^{poly(\|\Phi\|)}$, as required.
\end{proof}

In contrast to \EPAD, the complexity of deciding membership in \EPAG is \NP-complete. The problem inherits \NP-hardness from EPA (which itself is a consequence of the hardness of integer programming \cite{papadimitriou1981complexity}). 
A special case of \EPAG formulae are known to be in \NP due to their small-witness property \cite{IntegerProgGCD}. 
This is claimed to hold for all  \EPAG formulae, but not explicitly proven, so for completeness we include a proof of the following.

\begin{restatable}[Consequence of {\cite[Theorem~1]{IntegerProgGCD}}]{theorem}{EPAGisNP}\label{thm:EPAG_NP}
If $\exists\bm{x}\in\Z^n:\Phi(\bm{x})$ is a sentence in $\EPAG$ then there exists a solution $\bm{a}\in\Z^n$ to $\Phi$ such that $\|\bm{a}\|\le poly(\|\Phi\|)$.
\end{restatable}

Define \IPGCD (\emph{integer programming with GCD constraints}) as the subset of \EPAG consisting of true sentences of the form
\begin{equation}\label{eqn:IP-GCD_sentence}
	\exists\bm{x}\in\Z^n:\bigwedge_{i=1}^k f(\bm{x})\le 0\land \bigwedge_{i=1}^\ell \gcd\!\big(f_i(\bm{x}),g_i(\bm{x})\big)\approx_i d_i
\end{equation}
where each $f,g$ are linear polynomials with integer coefficients and indeterminates $\bm{x}$, each $d_i\in\N_+$, and each ${\approx_i}\in\{=,\le,\ne,\ge\}$. (Strictly speaking, \EPAG formulae do not contain the symbols $\ne$ or $\ge$, but atoms using these are logically equivalent to formulae in \EPAG as discussed in the proof of \cref{thm:EPAG_NP}.) Défossez \emph{et al.} prove the following.
\begin{theorem}[{\cite[Theorem~1]{IntegerProgGCD}}]\label{thm:IP-GCD_NP}
    If $\exists\bm{x}:\Phi(\bm{x})\in\IPGCD$ then there exists a solution $\bm{a}\in\Z^n$ to $\Phi$ such that $\|\bm{a}\|\le poly(\|\Phi\|)$.
\end{theorem}

\begin{proof}[Proof of \cref{thm:EPAG_NP}]
	We begin by transforming $\Phi$ into an equivalent formula $\Psi$ without negations and with atoms of the form
	\begin{equation}\label{eqn:PAG_NP}
		f(\bm{x})\le 0 \qquad\OR\qquad \gcd(f(\bm{x}),g(\bm{x}))\approx d
	\end{equation}
	where ${\approx}\in\{=,\ne,\ge,\le\}$ and $d\in\N_+$. We do this by using De Morgan's law to push negations inward until they apply directly to atoms, then using the following logical equivalences:
    \begin{equation}
        \begin{array}{rcl}
            a = b & \iff &  a-b\le 0\land b-a\le 0, \\
            \neg(a=b) & \iff & a-b+1\le0\lor b-a+1\le0,\\
            a\le b & \iff & a-b\le0, \\
            \neg(a\le b) & \iff & b-a+1\le0, \\
            \neg(\gcd(a,b)=d) & \iff & \gcd(a,b)\ne d, \\
            \neg(\gcd(a,b)\le d) & \iff & \gcd(a,b)\ge d+1.
        \end{array}
    \end{equation}
    Clearly $\|\Psi\|$ is linear in $\|\Phi\|$.

    Now, suppose $\bm{a}\in\Z^n$ is a solution to $\Phi$ and thus also to $\Psi$. The formula $\Psi$ consists of a finite number of nested conjuctive and disjunctive clauses. That is, 
    \begin{equation}
        \Psi=\bigwedge_{i=1}^{k_1}\bigvee_{j=1}^{\ell_{1,i}}\Theta_{1,i,j} \quad\text{where each}\quad \Theta_{1,i,j}=\bigwedge_{i=1}^{k_2}\bigvee_{j=1}^{\ell_{2,i}}\Theta_{2,i,j} \quad\text{where each}\quad \Theta_{2,i,j}=\cdots
    \end{equation}
    and so on, until eventually
    \begin{equation}
        \Theta_{m-1,i,j}=\bigwedge_{i=1}^{k_m}\bigvee_{j=1}^{\ell_{m,i}}\Theta_{m,i,j} \quad\text{where each $\Theta_{m,i,j}$ is an atom of the form \eqref{eqn:PAG_NP}.}
    \end{equation}
    Consider the outermost layer of disjunctive clauses. Because $\Psi(\bm{a})$ is true, for each $i\in[k_1]$ we have that $\Theta_{1,i,r_i}(\bm{a})$ is true for some $r_{i}\in[\ell_{1,i}]$. Therefore $\Psi'(\bm{a}):=\bigwedge_{i=1}^{k_1}\Theta_{1,i,r_i}(\bm{a})$ is also true. We continue moving inwards, replacing each disjunctive clause $\bigvee_{j=1}^{\ell_{x,i}}\Theta_{x,i,j}$ with a single formula $\Theta_{x,i,j}$ from that clause such that $\Theta_{x,i,j}(\bm{a})$ is true, until we obtain a formula~$\Upsilon$ which is a conjunction of atoms of the form \eqref{eqn:PAG_NP}. Note that $\exists\bm{x}:\Upsilon(\bm{x})$ is an $\IPGCD$ sentence with $\|\Upsilon\|=poly(\|\Phi\|)$. By \cref{thm:IP-GCD_NP} there exists a solution $\bm{b}\in\Z^n$ to $\Upsilon$ with $\|\bm{b}\|\le poly(\|\Upsilon\|)=poly(\|\Phi\|)$. By construction, $\bm{b}$ is also a solution to $\Psi$ (and thus to $\Phi$ as well), completing the proof.
\end{proof}

The exponential bound on solutions to \EPAGRP formulae from \cref{lem:EPAG-C_NEXP}, the logarithmic bound on coefficients needed to convert $n$-ary
to binary GCDs from \cref{thm:MAINC_k_aryGCD_to_binaryGCD}, and the polynomial bound on solutions to \EPAG formulae from \cref{thm:EPAG_NP} yields \cref{thm:MAINB_EPAGRP_poly_sol}.

\begin{proof}[Proof of \cref{thm:MAINB_EPAGRP_poly_sol}]
    Let $\exists\bm{x}\in\Z^n:\Phi(\bm{x})$ be a sentence in $\EPAGRP$ with 
    \begin{equation}
        \Phi(\bm{x}):= \Psi(\bm{x}) \land \bigwedge_{i=1}^k \gcd\!\big(f_{i,1}(\bm{x}),\dotsc,f_{i,p_i}(\bm{x})\big)\sim_i d_i.
    \end{equation}
    where $\Psi(\bm{x})$ is a quantifier-free \EPAG formula. Write $N=\|\Phi\|$. By \cref{lem:EPAG-C_NEXP} there exists a solution $\bm{a}\in\Z^n$ to $\Phi$ with $\|\bm{a}\|\le 2^{poly(N)}$. We will construct an \EPAG formula $\Theta$ from~$\Phi$ such that: $\bm{a}$ is a solution to $\Theta$, any solution $\bm{b}\in\Z^n$ to $\Theta$ is a solution to $\Phi$, and with $\|\Theta\|=poly(N)$.
    
    Let $i\in[k]$. There are two cases.\medskip
    \begin{description}
        \item[Case 1] Suppose $f_{i,j}(\bm{a})=0$ for all $j\in[p_i]$. There are two subcases.\smallskip
        \item[Subcase 1.1] If $\sim_i$ denotes $=$ then $\gcd\!\big(f_{i,1}(\bm{a}),\dotsc,f_{i,p_i}(\bm{a})\big)=d_i$ is false since $d_i>0$. This contradicts $\bm{a}$ being a solution to $\Phi$, so we can exclude this case.\smallskip
        \item[Subcase 1.2] If $\sim_i$ denotes $\le$ then we use the substitution
        \begin{equation}
            \gcd\!\big(f_{i,1}(\bm{x}),\dotsc,f_{i,p_i}(\bm{x})\big)\le d_i \quad\rightarrow\quad \bigwedge_{j=1}^{p_i}f_{i,j}(\bm{x})=0.
        \end{equation}
        By assumption, $\bm{a}$ is a solution to $\bigwedge_{j=1}^{p_i}f_{i,j}(\bm{x})=0$ in this case, and if $\bm{b}\in\Z^n$ is a solution to $\bigwedge_{j=1}^{p_i}f_{i,j}(\bm{x})=0$ then $\gcd\!\big(f_{i,1}(\bm{b}),\dotsc,f_{i,p_i}(\bm{b})\big)=0\le d_i$.\smallskip
        \item[Case 2] Suppose $f_{i,j}(\bm{a})\ne0$ for some $j\in[p_i]$. Without loss of generality we assume $f_{i,1}(\bm{a})\ne0$. By \cref{thm:MAINC_k_aryGCD_to_binaryGCD}, there exists
        \begin{equation}
            g_i(\bm{x}):=f_{i,2}(\bm{x}) + m_{i,3}f_{i,3}(\bm{x}) + \cdots + m_{i,p_i}f_{i,p_i}(\bm{x})
        \end{equation}
        such that 
        $\gcd\!\big(f_{i,1}(\bm{a}), g_i(\bm{a})\big) = \gcd\!\big(f_{i,1}(\bm{a}),\dotsc,f_{i,p_i}(\bm{a})\big)$ and where each
        \begin{equation}
    		m_{i,j} \ll \log^2|f_{i,1}(\bm{a})| \ \ \implies\ \  \|m_{i,j}\| \le poly\big(\log\|f_{i,1}(\bm{a})\|\big) \le poly\big(\log 2^{poly(N)}\big) = poly(N),
    	\end{equation}
        implying $\|g_i(\bm{x})\|\le poly(N)$. We consider two subcases.\smallskip
        \item[Subcase 2.1] If $\sim_i$ denotes $\le$ then we use the substitution
        \begin{equation}
            \gcd\!\big(f_{i,1}(\bm{x}),\dotsc,f_{i,p_i}(\bm{x})\big)\le d_i \quad\rightarrow\quad \gcd\!\big(f_{i,1}(\bm{x}), g_i(\bm{x})\big) \le d_i.
        \end{equation}
        Then $\bm{a}$ remains a solution to the latter. Conversely, if $\bm{b}\in\Z^n$ is a solution to $\gcd\!\big(f_{i,1}(\bm{x}), g_i(\bm{x})\big) \le d_i$ then there exists $c\le d_i$ such that $\gcd\!\big(f_{i,1}(\bm{b}), g_i(\bm{b})\big)=c$. By Bezout's identity there exists $s,t\in\Z$ such that
        \begin{equation}
            c \ = \ sf_{i,1}(\bm{b})+tg_i(\bm{b}) \ = \ sf_{i,1}(\bm{b}) +tf_{i,2}(\bm{b}) + sm_{i,3}f_{i,3}(\bm{b}) + \cdots + sm_{i,p_i}f_{i,p_i}(\bm{b}).
        \end{equation}
        But this implies $\gcd\!\big(f_{i,1}(\bm{b}),\dotsc,f_{i,p_i}(\bm{b})\big)\mid c\le d_i$. Hence $\bm{b}$ is a solution to 
        \begin{equation}
            \gcd\!\big(f_{i,1}(\bm{x}),\dotsc,f_{i,p_i}(\bm{x})\big)\le d_i.
        \end{equation}
        \item[Subcase 2.2] If $\sim_i$ denotes $=$ then we use the substitution
        \[
            \gcd\!\big(f_{i,1}(\bm{x}),\dotsc,f_{i,p_i}(\bm{x})\big)= d_i \quad\rightarrow\quad \gcd\!\big(f_{i,1}(\bm{x}), g_i(\bm{x})\big) = d_i\land\bigwedge_{j=1}^{p_i}\gcd(f_{i,j}(\bm{x}),d_i)=d_i.
        \]
        Note that $\gcd(f_{i,j}(\bm{x}),d_i)=d_i\ \iff\ d_i\mid f_{i,j}(\bm{x})$. As in the previous case, $\bm{a}$ remains a solution to the substituted formula. We also have that any solution $\bm{b}\in\Z^n$ to the substituted formula satisfies $\gcd\!\big(f_{i,1}(\bm{b}),\dotsc,f_{i,p_i}(\bm{b})\big)\mid d_i$ and $d_i\mid f_{i,j}(\bm{b})$ for all $j\in[p_i]$ and is thus also a solution to $\gcd\!\big(f_{i,1}(\bm{x}),\dotsc,f_{i,p_i}(\bm{x})\big)= d_i$.
    \end{description}
    
    Let $\Theta$ be the formula obtained from $\Phi$ by applying the above substitutions for each $i\in[k]$. Then $\exists\bm{x}\in\Z^n:\Theta(\bm{x})$ is a sentence in $\EPAG$ since $\bm{a}$ is a solution to $\Theta$, and $\|\Theta\|=poly(N)$. By \cref{thm:EPAG_NP}, there exists a solution $\bm{b}\in\Z^n$ to $\Theta$ with $\|\bm{b}\|\le poly(N)$. But any solution to~$\Theta$ is also a solution to $\Phi$.

    It follows immediately that deciding membership in \EPAGRP is in \NP: given a $\Phi$ of size~$N$ we can guess a solution $\|\bm{b}\|\le poly(N)$ and verify it by evaluating the truth value of~$\Phi(\bm{b})$ in $poly(N)$ time. \NP-completeness then follows from the hardness of \EPA.
\end{proof}

\section{ Inverse monoids, presentations and homomorphisms}\label{sec:InvMon}

A \textit{monoid} $\M$ is a set equipped with an associative binary operation and that contains an (identity) element $1$ such that  $1a=a1=a$ for all $a\in\M$. An \textit{idempotent} is an element $e\in\M$ satisfying $e^2=e$, and  we write $E(\M)$ for the set of idempotents in a monoid $\M$.

 If $\M$ and $\mathcal{N}$ are monoids with identity elements $1_\M$ and $1_{\mathcal{N}}$, respectively, then a 
  \textit{(monoid) morphism}
is a mapping $\phi:\M\rightarrow\mathcal{N}$ satisfying
\begin{equation}
	1_\M\phi=1_{\mathcal{N}} \qquad\AND\qquad (ab)\phi=(a\phi)(b\phi) \quad\FORALL\ a,b\in\M.
\end{equation}
An \textit{epimorphism} (resp. \textit{isomorphism}) is a surjective (resp. bijective) morphism. If there exists an isomorphism between monoids $\M$ and $\mathcal{N}$ then we say that they are \textit{isomorphic}.

In any monoid $\M$, a \textit{(generalised) inverse} of $a\in\M$ is an element $b\in\M$ such that $aba=a$ and $bab=b$. An \textit{inverse monoid} $\I$ is a monoid with an \textit{inversion} operation ${}^{-1}:\I\rightarrow\I$ satisfying
\begin{equation}\label{eqn:inv_mon_def}
	\big(a^{-1}\big)^{-1}=a, \quad aa^{-1}a=a, \quad\AND\quad\; aa^{-1}bb^{-1}=bb^{-1}aa^{-1} \quad\FORALL\ a,b\in\I.
\end{equation}
It follows from the definition that $a^{-1}aa^{-1}=a^{-1}$ for all $a\in\I$, implying $a^{-1}$ is the inverse of $a$, and that this inverse is unique. We also have that $E(\I)=\{aa^{-1}:a\in\I\}$. 
The set~$E(\I)$ has a natural partial order given by $e\le f \ \iff \ e=ef$. 

An \textit{inverse submonoid} is a subset $\J\subseteq\I$ that is itself an inverse monoid with $1_\J=1_\I$. If $X\subseteq\I$ then we write $\Inv\langle X\rangle$ for the 
smallest inverse submonoid of $\I$ that contains $X$ (i.e., the intersection of all submonoids of $\I$ that contain $X$), called the \textit{inverse submonoid generated by $X$}. If $\Inv\langle X\rangle=\I$ or, equivalently,
\begin{equation}
	\I=\{x_1\cdots x_k:x_1,\dotsc,x_k\in\swfi{X}\} \quad\WHERE\quad \swfi{X}=X\cup X^{-1}=X\cup\{x^{-1}:x\in X\},
\end{equation}
then we say $X$ is a \textit{generating set} for $\I$. If $\I$ and $\J$ are inverse monoids and $\phi:\I\rightarrow\J$ a morphism then $\phi$ satisfies ${(a\phi)^{-1}=(a^{-1})\phi}$ for all $a\in\I$.
We have $\im(\phi)$ is an inverse submonoid of $\J$, and if $X$ generates $\I$ then $X\phi$ generates $\im(\phi)$.

Let $X$ be a set of symbols called an \textit{alphabet}. The \textit{free monoid} on $X$ is the set $X^*$ of words over $X$ under concatenation with $1\in X^*$ being the empty word. We write $|w|$ for the length of a word $w\in X^*$. Any monoid $\M$ and mapping $\alpha:X\rightarrow\M$ induces a morphism
\begin{equation}
	\phi:X^*\rightarrow\M \quad\GIVENBY\quad (x_1\cdots x_k)\phi = (x_1\alpha)\cdots(x_k\alpha) \quad\FORALL\ x_1,\dotsc,x_k\in X.
\end{equation}
For an alphabet $X$ we write $X^{-1}=\{x^{-1}:x\in X\}$, which is an alphabet disjoint from $X$ and consisting of formal inverses of the letters in $X$. For an alphabet $X$ we write $\swfi{X}=X\cup X^{-1}$.

A \emph{congruence} on a monoid $\M$ is an equivalence relation $\rho\subseteq \M\times \M$ for which 
\begin{equation}
    (a_1,a_2),(b_1, b_2)\in\rho \quad\implies\quad (a_1b_1, a_2b_2)\in\rho \qquad\FORALL\ \ a_1,a_2,b_1,b_2\in\M
\end{equation}
If $\rho$ is a congruence on $\M$ then the quotient $\M/\rho=\big\{[a]:a\in\M\big\}$ is a monoid under the induced operation $[a][b]=[ab]$ with identity $[1]\in\M/\rho$. 
If $R$ is a relation on $\M$ we write $R^\sharp$ for the smallest congruence on $\M$ containing $R$ (which equals the intersection of all congruences on $\M$ that contain $R$),  called the \textit{congruence generated by $R$}. 

If $\rho$ is a congruence on $\M$ then the mapping $\rho^\natural:\M\rightarrow\M/\rho:a\mapsto[a]$ is an epimorphism. 
The \textit{kernel} of a morphism $\phi:\M\rightarrow\mathcal{N}$ is the congruence $\ker(\phi) = \{(a,b)\in\M:a\phi=b\phi\}$. The \textit{fundamental homomorphism theorem (FHT)} for monoids is as follows.
\begin{theorem}[FHT]\label{thm:fht}
	Let $\M$ and $\mathcal{N}$ be monoids, let $\phi:\M\rightarrow\mathcal{N}$ be a morphism, and let~$\rho$ be a congruence on $\M$. If $\rho\subseteq\ker(\phi)$ then the mapping $\psi:\M/\rho\rightarrow \mathcal{N}:[a]\mapsto a\phi$ is a morphism with $\im(\psi)=\im(\phi)$.
\end{theorem}

Define the notation
\begin{equation}
	\big(x^{-1}\big)^{-1}=x^{-1} \qquad\AND\qquad (y_1\cdots y_k)^{-1}=y_k^{-1}\cdots y_1^{-1}
\end{equation}
for all $x\in X$ and $y_1,\dotsc,y_k\in \swfi{X}$. Define the relation  $\varrho$ on $(\swfi{X})^*$ as
\begin{equation}\label{eqn:tau_def}
	\varrho = \Big\{(uu^{-1}u,u):u\in (\swfi{X})^*\Big\}\cup\Big\{(uu^{-1}vv^{-1},vv^{-1}uu^{-1}):u,v\in (\swfi{X})^*\Big\}.
\end{equation}
We call $\varrho^\sharp$ the \emph{Wagner congruence} on $X$. Observe that  $\varrho$ describes the equivalences necessary for a quotient of $(\swfi{X})^*$ to be an inverse monoid; consequently, $(\swfi{X})^*/\rho$ is an inverse monoid (with $[w]^{-1}=[w^{-1}]$) if and only if $\varrho\subseteq\rho$.

Let $\I$ be an inverse monoid. If is $R$ a relation on $(\swfi{X})^*$ and 
$\I$ is isomorphic to $(\swfi{X})^*/(\varrho\cup R)^\sharp$ 
then we say that $\I$ has an \textit{inverse monoid presentation} $\Inv\pres{X}{R}$.

An inverse monoid $\I$ is \textit{finitely presented} if both $X$ and $R$ are finite.
We define the size of a finite inverse monoid presentation \[\big\|\Inv\pres{X}{R}\big\|=|X|+\sum_{(u,v)\in R}\big(|u| + |v|\big).\]

An \textit{inverse monoid equation (without constants)} is a pair $(u,v)\in(\swfi{X})^*\times(\swfi{X})^*$. In the context of equations we call the letters in $X$ \textit{variables}. A solution to $(u,v)$ over an inverse monoid $\I$ is a mapping $\sigma:X\rightarrow\I$ such that the morphism $\ol\sigma$ induced by
\begin{equation}
	\sigma':\swfi{X}\rightarrow\I:\left\{\begin{array}{ll}
		x\mapsto x\sigma \\
		x^{-1}\mapsto (x\sigma)^{-1}
	\end{array}\right.
\end{equation}
satisfies $u\ol\sigma=v\ol\sigma$. (Note that mapping $\sigma'$ is uniquely determined by $\sigma$ since the inverse of each element $x\sigma\in\I$ is unique, so we can refer to $\ol\sigma$ as the morphism induced by $\sigma$ without confusion.) In other words, the mapping $\sigma$ specifies elements of $\I$ that can be substituted for the variables in $X$ so that the equation $u=v$ holds in $\I$. A \textit{system} of inverse monoid equations is a relation $R\subseteq(\swfi{X})^*\times(\swfi{X})^*$. A mapping $\sigma:X\rightarrow\I$ is a solution to $R$ over $\I$ if it is a solution over $\I$ for all inverse monoid equations $(u,v)\in R$.

We can treat the relations $R$ in an inverse monoid presentation $\Inv\pres{X}{R}$ as a system of equations by interpreting $X$ as a set of variables. Doing so reduces the problem of testing whether a mapping $\sigma:\Inv\pres{X}{R}\rightarrow\I$ defined on generators is a morphism to determining if $\sigma$ is a solution to $R$ over $\I$. This is summarised in the following proposition (a consequence of \cref{thm:fht}).
 Note that in \cref{thm:von-dyck}, $[u]$ represents the element of $\I$ corresponding to the word $u\in(\swfi{X})^*$

\begin{restatable}{proposition}{vondyck}\label{thm:von-dyck}
	Let $\I=\Inv\pres{X}{R}$ and $\J$ be inverse monoids.
	\begin{enumerate}
		\item[(i)]\label{thm:von-dyck1} If $\sigma:X\rightarrow\J$ is a solution to $R$ over $\J$ and $\ol\sigma:(\swfi{X})^*\rightarrow\J$ is the morphism induced by~$\sigma$, then $\phi:\I\rightarrow\J:[u]\mapsto u\ol\sigma$ is a morphism. Furthermore, if $X\sigma$ generates~$\J$ then $\phi$ is an epimorphism.
		\item[(ii)]\label{thm:von-dyck2} Conversely, if $\phi:\I\rightarrow\J$ is a morphism then $\sigma:X\rightarrow\J:x\mapsto[x]\phi$ is a solution to $R$ over~$\J$. Furthermore, if $\phi$ is an epimorphism then $X\sigma$ generates $\J$.
	\end{enumerate}
\end{restatable}

\begin{proof}
	 (i): 
	If $u=x_1^{n_1}\cdots x_k^{n_k}$ for some $x_1,\dotsc,x_k\in X$ and ${n_1,\dotsc,n_k\in\{-1,1\}}$ then
	\begin{equation}
		\big(u^{-1}\big)\ol\sigma = \big(x_k^{-n_k}\cdots x_1^{-n_1}\big)\ol\sigma = \big(x_k^{-n_k}\sigma'\big)\cdots \big(x_1^{-n_1}\sigma'\big) = (x_k\sigma)^{-n_k}\cdots(x_1\sigma)^{-n_1} = (u\ol\sigma)^{-1}.
	\end{equation}
	It then follows from $\J$ being an inverse monoid that $\varrho\subseteq\ker(\ol\sigma)$. We also have that $u\ol\sigma=v\ol\sigma$ for all $(u,v)\in R$ because $\sigma$ is a solution to $R$ and thus $R\subseteq\ker(\ol\sigma)$. Therefore $(R\cup\varrho)^\sharp\subseteq\ker(\ol\sigma)$ and $\phi$ is a morphism by \cref{thm:fht}. If in addition $X\sigma$ generates $\J$ then $\im(\phi)=\im(\ol\sigma)=\J$, as required.

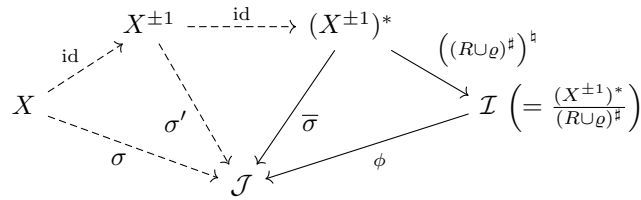
\begin{figure}[ht]
	\begin{center}
		\begin{tikzcd}[column sep=small, row sep=small]
			& & \swfi{X} \arrow[to=3-4, dashed, "\displaystyle{\sigma'}"'] \arrow[to=1-5, dashed, "\scriptstyle{\id}"] 
			& & (\swfi{X})^* \arrow[to=3-4, "\displaystyle{\ol\sigma}"] \arrow[to=2-7, "\scriptstyle{\big((R\cup\varrho)^\sharp\big)^\natural}"] & & \\
			X \arrow[to=1-3, dashed, "\scriptstyle{\id}"]  \arrow[to=3-4, dashed, "\displaystyle{\sigma}"'] & & & & & & \I \arrow[to=3-4, "\phi"] & \hspace{-6mm}\left(=\frac{(\swfi{X})^*}{(R\cup\varrho)^\sharp}\right) \\ 
			& & & \J
		\end{tikzcd}
	\end{center}
	\caption{Commutative diagram illustrating the proof of \cref{thm:von-dyck}\eqref{thm:von-dyck1}. Dashed arrows represent set mappings. Solid arrows represent morphisms. Arrows marked with $\id$ are identity mappings.}
\end{figure}

\noindent
	(ii): First, recall that $[x^{-1}]=[x]^{-1}$ and $[x]^{-1}\phi=\big([x]\phi\big)^{-1}$ for all $x\in X$. Hence if $u\in (\swfi{X})^*$ then there exist $x_1,\dotsc,x_k\in X$ and $n_1,\dotsc,n_k\in\{-1,1\}$ such that
	\begin{equation}\label{eqn:von-dyck2.1}
		[u]\phi = \big[x_1^{n_1}\cdots x_k^{n_k}\big]\phi = \big([x_1]\phi\big)^{n_1}\cdots \big([x_k]\phi\big)^{n_k} = (x_1\sigma)^{n_1}\cdots (x_k\sigma)^{n_k}.
	\end{equation}
	Let $(u,v)\in R$ and let $\ol\sigma:(\swfi{X})^*$ be the morphism induced by $\sigma$. Then by \eqref{eqn:von-dyck2.1} there exists $x_1,\dotsc,x_k\in X$ and $n_1,\dotsc,n_k\in\{-1,1\}$ such that
	\begin{equation}
		[u]\phi = (x_1\sigma)^{n_1}\cdots (x_k\sigma)^{n_k} = \big(x_1^{n_1}\sigma'\big)\cdots\big(x_k^{n_k}\sigma'\big) = u\ol\sigma.
	\end{equation}
	Similarly, we have $[v]\phi=v\ol\sigma$. But $[u]=[v]$ because $(u,v)\in R$ and thus $u\ol\sigma=[u]\phi=[v]\phi=v\ol\sigma$. Suppose that $a\in\J$ and $\phi$ is an epimorphism. Then there exists $[u]\in\I$ such that $[u]\phi=a$, and by \eqref{eqn:von-dyck2.1} there exists $x_1,\dotsc,x_k\in X$ and $n_1,\dotsc,n_k\in\{-1,1\}$ such that
	\[a = [u]\phi = (x_1\sigma)^{n_1}\cdots(x_k\sigma)^{n_k}. \qedhere \]
\end{proof}

\subsection{The bicyclic monoid}

The \textit{bicyclic monoid} $\B$ is the inverse monoid given by the inverse monoid presentation $\Inv\pres{a}{aa^{-1}=1}$. It is immediate that every word in $\{a,a^{-1}\}^*$ can be reduced to a unique word of the form $a^{-m}a^n$ for some $m,n\in\N$ using the relation $aa^{-1}=1$, and that the (reduced) product of two such words is given by $a^{-m}a^n a^{-p}a^q=a^{-(m-n+t)}a^{q-p+t}$ where $t=\max(n,p)$.
As such, we can also characterise $\B$ as the inverse monoid with underlying set $\N\times\N$, binary operation
\begin{equation}
	(m,n)(p,q) = \big(m-n+\max(n,p),q-p+\max(n,p)\big),
\end{equation}
identity $1=(0,0)$, and inversion $(m,n)^{-1}=(n,m)$. Consequently, if $(m,n)\in\B$ then
\begin{equation}\label{eqn:idempotent_B}
	(m,n)(m,n)^{-1} = (m,m) \qquad\AND\qquad (m,n)^{-1}(m,n) = (n,n)
\end{equation}
and thus $E(\B)=\{(m,m):m\in\N\}$ with, counterintuitively, natural order
\begin{equation}\label{eqn:order_E(B)}
	(m,m) \le (n,n) \quad\iff\quad  (m,m)=(m,m)(n,n) \quad\iff\quad m\ge n.
\end{equation}
Define $\lambda$ and $\rho$ as the left and right projections of $\B$ onto $\N$; that is,
\begin{equation}
	\lambda(m,n) = m \qquad\AND\qquad \rho(m,n) = n \qquad\FORALL\ (m,n)\in\N\times\N.
\end{equation}
Define the difference mapping $\delta:\B\rightarrow\Z:(m,n)\mapsto m-n$. A short calculation will show that $\delta$ is a morphism (where the operation on $\Z$ is addition).

The morphisms into $\B$ that we are studying are defined on generators. Hence the necessary and sufficient conditions for these morphisms to be surjective are tied directly to the necessary and sufficient conditions for a subset of $\B$ to be a generating set. The subsemigroups of $\B$ were studied in detail in \cite{ruskuc2005subsemigroups}.  \cite[Theorem~8.1]{ruskuc2005subsemigroups} describes the subsemigroup generated by an arbitrary set $B\subseteq\B$, and is necessarily very technical and more general than what is needed for our purposes. For this reason we 
restate, in simpler terms, the consequence of that theorem which is relevant here, with a self-contained proof.

\begin{theorem}\label{thm:gen_set_B}
	Let $B=\{b_1,\dotsc,b_k\}\subseteq\B$. Then $\Inv\langle B\rangle=\B$ if and only if:
	\begin{enumerate}
		\item[(i)]\label{thm:gen_set_B_i} there exists $m>0$ such that $(m,0)\in B$ or $(0,m)\in B$, and
		\item[(ii)]\label{thm:gen_set_B_ii} $\gcd\!\big(\delta(b_1),\dotsc,\delta(b_k)\big)=1$.
	\end{enumerate}
\end{theorem}
\begin{proof}
	($\Rightarrow$): If $\Inv\langle B\rangle=\B$ then there exist elements $x_1,\dotsc,x_\ell\in \swfi{B}$ such that 
	\begin{equation}\label{eqn:product_for_(1,0)}
		x_1\cdots x_\ell = (1,0).
	\end{equation}
	Without loss of generality, we can assume that each $x_i\ne 1$. For all $b,c\in\B$ we have
	\begin{equation}
		\rho(bc) \ = \ \rho(c) - \lambda(c) + \max\big(\rho(b),\lambda(c)\big) \ \ge \ \rho(c) - \lambda(c) + \lambda(c) \ = \ \rho(c);
	\end{equation}
	in particular, we have $0 = \rho(1,0) = \rho(x_1\cdots x_\ell) \ge \rho(x_\ell)$ and thus $\rho(x_\ell)=0$. But $\lambda(x_\ell)\ne 0$ since $x_\ell\ne 1$, so $x_\ell=(m,0)$ for some $m>0$. Therefore $x_\ell\in B$ or $x_\ell^{-1}\in B$ satisfies~\eqref{thm:gen_set_B_i}. 
	
	For (ii), note that each $\delta(x_i)=\pm\delta(b)$ for some $b\in B$. By \eqref{eqn:product_for_(1,0)} this implies
	\begin{equation}
		1 \ = \ \delta(1,0) \ = \ \delta(x_1\cdots x_l) \ = \ \delta(x_1) + \cdots + \delta(x_\ell) \ = \ m_1\delta(b_1) + \cdots + m_k\delta(b_k)
	\end{equation}
	for some $m_1,\dotsc,m_k\in\Z$ and thus $\gcd\!\big(\delta(b_1),\dotsc,\delta(b_k)\big)=1$.
	
	($\Leftarrow$): It suffices to show $(0,1)\in \Inv\langle B\rangle$. By \eqref{thm:gen_set_B_ii}, there exist $m_1,\dotsc,m_k\in\Z$ such that
	\begin{equation}
		m_1\delta(b_1) + \cdots + m_k\delta(b_k) = 1,
	\end{equation}
	so the element $x=b_1^{m_1}\cdots b_k^{m_k}\in\Inv\langle B\rangle$ is such that $\delta(x)=1$. That is, $x = (n+1,n)$ for some $n\in\N$. By \eqref{thm:gen_set_B_i}, there exists $m\ge1$ such that $(0,m)\in B$ or $(m,0)\in B$, implying $(0,m)\in\Inv\langle B\rangle$. Hence
	\begin{equation}
		(0,nm+1) \ = \ (0,nm)(n,n+1) \ = \ (0,m)^n x^{-1} \ \in \ \Inv\langle B\rangle.
	\end{equation}
	Therefore
	\[
	(0,1) \ = \ (0,nm+1)(nm,0) \ = \ (0,nm+1)(0,m)^{-n}  \ \in \ \Inv\langle B\rangle. \qedhere
	\]
\end{proof}

 \section{\NP-hardness}\label{sec:NPhard}

Let 
\Epi denote the 
decision problem: on input a finite inverse monoid presentation $\Inv\pres{X}{R}$ for an inverse monoid $\I$, decide whether there exists an epimorphism $\I\rightarrow\B$.

We prove that the problem \Epi is \NP-hard by showing that the \NP-hard problem 3-SAT is polynomial time reducible to \Epi. This is done by constructing, from an arbitrary 3-CNF formula $\Phi$, an inverse monoid presentation $\cP(\Phi)$ that admits an epimorphism onto $\B$ if and only if $\Phi$ is satisfiable. This technique 
follows 
recent work \cite{elderVAbel} studying the complexity of the epimorphism problem for certain classes of groups. 
The form of presentation $\cP(\Phi)$ is inspired by work of Lopatin and  Rybalov on hardness of testing the feasibility of equations (with constants) over $\B$ \cite{lopatin2025bicyclicequations}.

A 3-CNF formula is an open formula $\Phi(\Theta)$ in propositional logic over a set of free variables $\Theta=\{\theta_1,\dotsc,\theta_M\}$ with the form
\begin{equation}\label{eqn:3CNF_formula}
	\Phi(\Theta) = \bigwedge_{i=1}^{N} \big(\alpha_{i,1}\lor\alpha_{i,2}\lor\alpha_{i,3}\big) \quad\text{where each}\quad \alpha_{i,j}=\theta_k \ \ \OR \ \ \neg\theta_k \quad\FORSOME\ 1\le k\le M.
\end{equation}
A 3-CNF formula $\Phi(\Theta)$ is said to be satisfiable if there exists an assignment ${\sigma:\Theta\rightarrow\{0,1\}}$ (where $0$, $1$ represent logical FALSE and logical TRUE, respectively) such that ${\Phi\big(\Theta\sigma\big)=1}$. For notational convenience we write
\begin{equation}
	\alpha_{i,j}\sigma = \left\{\begin{array}{ll}
		\theta_k\sigma & \text{if $\alpha_{i,j}=\theta_k$} \\
		\neg(\theta_k\sigma) & \text{if $\alpha_{i,j}=\neg\theta_k$,}
	\end{array}\right.
\end{equation}
and note that
\begin{equation}\label{eqn:sat_assignment}
	\Phi\big(\Theta\sigma\big)=1 \quad\iff\quad \text{for all $1\le i\le j$ there exists $1\le j\le 3$ such that $\alpha_{i,j}\sigma=1$.}
\end{equation}
The decision problem 3-SAT is as follows: on input a 3-CNF formula, decide if it is satisfiable. It is well-known that this problem is \NP-complete and thus \NP-hard.

Given a 3-CNF formula $\Phi(\Theta)$ as in \eqref{eqn:3CNF_formula}, define the alphabet
\begin{equation}
	\X(\Phi) = \{x_1,\dotsc,x_M\}\cup\{y_1,\dotsc,y_M\}\cup\{u_1,\dotsc,u_N\}\cup\{s\}.
\end{equation}
Define $\R(\Phi)$ as the following set of relations on $(\swfi{\X(\Phi)})^*$ (where we write $u=v$ to indicate $(u,v)\in\R(\Phi)$):
\begin{align}
		\label{R1}\tag{R1} x_ix_i^{-1} = y_iy_i^{-1}  &= u_ju_j^{-1}  = ss^{-1}  = 1, \\
		\label{R2}\tag{R2} x_iy_i &= s, \\
		\label{R3}\tag{R3} (x_i^{-1}x_i)(s^{-1}s)(y_i^{-1}y_i) &= (x_i^{-1}x_i)(y_i^{-1}y_i), \\
		\label{R4}\tag{R4} t_{j,1}t_{j,2}t_{j,3} &= u_js \qquad\WHERE\qquad t_{j,k} = \left\{\begin{array}{ll}
			x_i & \text{if $\alpha_{j,k}=\theta_i$} \\
			y_i & \text{if $\alpha_{j,k}=\neg\theta_i$,}
		\end{array}\right.
\end{align}
for all $1\le i\le M$ and $1\le j\le N$. Define $\cP(\Phi)=\Inv\pres{\X(\Phi)}{\R(\Phi)}$.

\begin{lemma}\label{thm:3-SAT_iff_Epi(I,B)}
	A 3-CNF formula $\Phi(\Theta)$ is satisfiable iff there exists an epimorphism ${\cP(\Phi)\rightarrow\B}$.
\end{lemma}
\begin{proof}
	($\Rightarrow$): Suppose there exists a satisfying assignment $\sigma:\Theta\rightarrow\{0,1\}$. Treating logical $0$ and $1$ as integers, define $\tau:\X(\Phi)\rightarrow\B$ by 
	\begin{equation}
		\setlength{\arraycolsep}{4pt}
		\begin{array}{rclcrcl}
			s\tau & = & (0,1), & \qquad & y_i\tau & = & \big(0,\neg(\theta_i\sigma)\big), \quad\AND \\[2pt]
			x_i\tau & = & (0,\theta_i\sigma), & & u_j\tau & = & (0,\alpha_{j,1}\sigma+\alpha_{j,2}\sigma+\alpha_{j,3}\sigma-1),
		\end{array}
	\end{equation}
	for all $1\le i\le M$ and $1\le j\le N$. For brevity, we will write $\wh{z}=z\tau$ for $z\in \X(\Phi)$. We have that $\X(\Phi)\tau$ generates $\B$ since $(0,1)=\wh{s}\in\X(\Phi)\tau$, so we need only show that $\tau$ is a solution to $\R(\Phi)$ over $\B$ for the result to follow from \cref{thm:von-dyck}(i). 
	
	For \eqref{R1}, observe that $\alpha_{i,1}\sigma+\alpha_{i,2}\sigma+\alpha_{i,3}\sigma-1\ge0$ for all $1\le i\le M$ by \eqref{eqn:sat_assignment}. Therefore for all $z\in\X(\Phi)$ we have $\wh{z}=(0,n)$ for some $n\in\N$ and $\tau$ is a solution to \eqref{R1}. For each $1\le i\le M$, either $\wh{x}_i=(0,1)$ and $\wh{y}_i=(0,0)$, or $\wh{x}_i=(0,0)$ and $\wh{y}_i=(0,1)$. This implies $\tau$ is a solution to \eqref{R2}. It also implies that either
	\begin{equation}
		(\wh{x}_i^{-1}\wh{x}_i)(\wh{s}^{-1}\wh{s}) = (1,1)(1,1) = (1,1) = \wh{x}_i^{-1}\wh{x}_i \quad\text{or similarly}\quad (\wh{s}^{-1}\wh{s})(\wh{y}_i^{-1}\wh{y}_i) = \wh{y}_i^{-1}\wh{y}_i,
	\end{equation} 
	and hence $\tau$ is a solution to \eqref{R3}. Finally, for equation \eqref{R4}, note that
	\begin{equation}
		\wh{t}_{i,j} = \left\{\begin{array}{ll}
			(0,\theta_k\sigma) & \text{if $\alpha_{i,j}=\theta_k$} \\
			\big(0,\neg(\theta_k\sigma)\big) & \text{if $\alpha_{i,j}=\neg\theta_k$.}
		\end{array}\right.
	\end{equation}
	In either case, we have $\wh{t}_{i,j}=(0,\alpha_{i,j}\sigma)$ and thus for all $1\le i\le N$
	\begin{equation}
		 \wh{u}_i\wh{s} \ = \ (0,\alpha_{j,1}\sigma+\alpha_{j,2}\sigma+\alpha_{j,3}\sigma-1)(0,1) \ = \ (0,\alpha_{j,1}\sigma)(0,\alpha_{j,2}\sigma)(0,\alpha_{j,3}\sigma) \ = \ \wh{t}_{i,1}\wh{t}_{i,2}\wh{t}_{i,3}.
	\end{equation}
	
	($\Leftarrow$): Suppose $\phi:\cP(\Phi)\rightarrow\B$ is an epimorphism. Define $\tau:\X(\Phi)\rightarrow\B:z\mapsto[z]\phi$ where we again write $\wh{z}=z\tau$ for any $z\in\X(\Phi)$. We will show that $\sigma:\Theta\rightarrow\{0,1\}:\theta_i\mapsto\rho(\wh{x}_i)$ is a satisfying assignment for $\Phi(\Theta)$ by first identifying the possible maps for $\tau$.
	
	By \cref{thm:von-dyck}(ii) $\tau$ is a solution to $\X(\Phi)$ over $\B$. By \eqref{eqn:idempotent_B}, if $(m,n)(m,n)^{-1}=1$ then $m=0$. Hence equation \eqref{R1} implies that for $1\le i\le M$ and $1\le j\le N$ we have
	\begin{equation}
		\wh{s}=(0,n), \qquad \wh{x}_i=(0,p_i), \qquad \wh{y}_i=(0,q_i), \qquad\AND\qquad \wh{u}_j=(0,m_j)
	\end{equation}
	for some $n,p_i,q_i,m_i\in\N$. Equation \eqref{R2} implies that $p_i+q_i=n$ for all $1\le i\le M$. It follows from the order on~$E(\B)$ and equation \eqref{R3} that
	\begin{equation}
		\setlength{\arraycolsep}{4pt}
		\begin{array}{rcl}
			(p_i+q_i,p_i+q_i) \ = \ \wh{s}^{-1}\wh{s} \ = \ (\wh{x}_i^{-1}\wh{x}_i)(\wh{s}^{-1}\wh{s})(\wh{y}_i^{-1}\wh{y}_i) & = & (\wh{x}_i^{-1}\wh{x}_i)(\wh{y}_i^{-1}\wh{y}_i) \\
			& = & \big({\max(p_i,q_i)},\max(p_i,q_i)\big).
		\end{array}
	\end{equation}
	Therefore either $(p_i,q_i)=(0,n)$ or $(p_i,q_i)=(n,0)$. Equation \eqref{R4} then implies that there exists $0\le \ell\le 3$ such that $(0,\ell n) = \wh{t}_{j,1}\wh{t}_{j,2}\wh{t}_{j,3} = \wh{u}_j\wh{s} = (0,m_j+n)$ since each $\wh{t}_{j,k}=(0,p_i)$ or $(0,q_i)$ for some $1\le i\le M$. That is, each $m_j=r_jn$ for some~$0\le r_j\le 2$. We have shown that for all $z\in\X(\Phi)$ we have $\wh{z}=(0,kn)$ for some $k\in\N$. But $\phi$ is surjective, so $\X(\Phi)\tau$ generates $\B$ by \cref{thm:von-dyck}(ii) and $n=1$ by \cref{thm:gen_set_B}(ii). In summary, we have $\wh{s}=(0,1)$,
	\begin{equation}
		 (\wh{x}_i,\wh{y}_i)=\big((0,1),(0,0)\big) \ \ \OR \ \ \big((0,0),(0,1)\big) \qquad\FORALL\ 1\le i\le M,
	\end{equation}
	and $\wh{u}_j=(0,r_j)$ where $0\le r_j\le 2$ for all $1\le j\le N$.
	
	We can now show that $\sigma$ is a satisfying assignment for $\Phi(\Theta)$. Let $1\le j\le N$. By \eqref{R4}
	\begin{equation}
		\wh{t}_{j,1}\wh{t}_{j,2}\wh{t}_{j,3} = \wh{u}_j\wh{s} = (0,r_j+1) \qquad\implies\qquad \wh{t}_{j,k}=(0,1) \ \ \FORSOME \ \ 1\le k\le 3.
	\end{equation}
	If $t_{j,k}=x_i$ for some $1\le i\le M$ then $\theta_i\sigma=\rho(\wh{x}_i)=1$ and $\alpha_{j,k}\sigma=\theta_i\sigma=1$. If instead $t_{j,k}=y_i$ for some $1\le i\le M$ then $\theta_i\sigma=\rho(\wh{x}_i)=0$ and $\alpha_{j,k}\sigma=\neg(\theta_k\sigma)=1$, as required.
\end{proof}

The presentation $\cP(\Phi)$ has can be constructed in time linear in $\|\Phi\|$, so we have proved:

\begin{proposition}\label{thm:epi_is_NP-hard}
	The decision problem \Epi is \NP-hard.
\end{proposition}

 \section{Containment in \NP}\label{sec:inNP}

 For a mapping $\phi:\Inv\pres{X}{R}\rightarrow\B$ defined on the generators $X$ to be an epimorphism, it must be  a solution to the system of inverse monoid equations $R$, and the image of the generators must be a generating set for $\B$. 
The binary operation on $\B$ can be realised 
in terms of linear arithmetic on integers, so equations over $\B$ have a natural interpretation in Presburger arithmetic. However, a direct transformation of the equations in $R$ into PA would yield an exponential increase in size. This is a consequence of the branching implied by the $\max$ function in the operation on $\B$. Instead, we use the technique of \textit{triangulation}, that is, splitting equations involving four or more variables into equations of the form $x=y$ or $x=yz$ by introducing new variables. 
This method has been used to show that testing the feasibility of equations (with constants) over $\B$ is in \NP \cite{lopatin2025bicyclicequations}, and is standard in solving equations over monoids and groups. Once we have reduced testing a solution for $R$ to a PA formula, we add additional constraints based on \cref{thm:gen_set_B} to ensure any resulting morphism is surjective. This reduction to a \EPAGRP formula yields the \NP bound.

Let $R$ be a system of inverse monoid equations with variable set $X$. Recall that a mapping $\sigma:X\rightarrow\B$ is a solution to $R$ over $\B$ if the morphism $\ol\sigma$ induced by $\sigma':\swfi{X}\rightarrow\B$ where $x\sigma'=x\sigma$ and $x^{-1}\sigma'=(x\sigma)^{-1}$ satisfies $u\ol\sigma=v\ol\sigma$ for all $(u,v)\in R$. However, a direct transformation of $R$ into a PA formula would sever the implied connection between the variables $x$ and $x^{-1}$. To maintain this connection, we transform $R$ into a new system, without formal inverses, over a variable set $X\cup\wh{X}$ and encode the inverse relationship between $x$ and~$\wh{x}$ in the equations themselves.

Define a new variable set ${\wh{X}=\{\wh{x}:x\in X\}}$ disjoint from and in one-to-one correspondence with $X$. Let $\theta:(\swfi{X})^*\rightarrow \big(X\cup\wh{X}\big)^*$ be the morphism induced by the mapping
\begin{equation}
	\swfi{X}\rightarrow \big(X\cup\wh{X}\big)^*:\left\{\begin{array}{l}
		x\mapsto x \\
		x^{-1} \mapsto \wh{x}.
	\end{array}\right.
\end{equation}
Define the following system of equations with variable set $X\cup\wh{X}$:
\begin{equation}
	\Theta(R) = \big\{(x\wh{x}x,x),(\wh{x}x\wh{x},\wh{x}):x\in X\big\}\cup\big\{(u\theta,v\theta):(u,v)\in R\big\}.
\end{equation} 
In essence, the following lemma states that a solution for $R$ can be obtained from a solution to $\Theta(R)$ in the obvious way, and vice versa.

\begin{lemma}\label{lem:rem_inv_ele_in_eqns}
	Let $\I$ be an inverse monoid and $R\subseteq (\swfi{X})^*\times(\swfi{X})^*$. If $\tau:X\rightarrow \I$ is a solution to~$R$ then there exists a solution $\sigma:X\cup\wh{X}\rightarrow\I$ to $\Theta(R)$ such that $x\sigma=x\tau$ for all $x\in X$. Conversely, if $\sigma$ is a solution to~$\Theta(R)$ then~$\sigma|_X$ is a solution to $R$.
\end{lemma}

\begin{proof}
	Suppose $\tau:X\rightarrow \I$ is a solution to $R$. Define
	\begin{equation}
		\sigma:X\cup\wh{X}\rightarrow\I:\left\{\begin{array}{l}
			x\mapsto x\tau \\
			\wh{x}\mapsto (x\tau)^{-1} \text{ ($=x^{-1}\tau$ since $\tau$ is a solution to $R$),}
		\end{array}\right.
	\end{equation}
	and note that $\theta\sigma=\tau$ for all $x\in\swfi{X}$. It follows that the morphism $\ol\sigma:\big(X\cup\wh{X}\big)^*\rightarrow\I$ induced by $\sigma$ satisfies $(u\theta)\ol\sigma=u\ol\tau=v\ol\tau=(v\theta)\ol\sigma$ for all $(u,v)\in R$. We also have
	\begin{equation}
		(x\wh{x}x)\ol\sigma = (x\tau)(x\tau)^{-1}(x\tau) = (x\tau) = x\ol\sigma
	\end{equation}
	and similarly $(\wh{x}x\wh{x})\ol\sigma=\wh{x}\ol\sigma$. Therefore $\sigma$ is a solution to $\Theta(R)$.
	
	For the converse, suppose $\sigma:X\cup\wh{X}\rightarrow\I$ is a solution to $\Theta(R)$. Then
	\begin{equation}
		(x\sigma)(\wh{x}\sigma)(x\sigma) = (x\wh{x}x)\ol\sigma = x\ol\sigma = x\sigma \qquad\AND\qquad (\wh{x}\sigma)(x\sigma)(\wh{x}\sigma) = (\wh{x}x\wh{x})\ol\sigma = \wh{x}\ol\sigma = \wh{x}\sigma
	\end{equation}
	imply that $\wh{x}\sigma=(x\sigma)^{-1}$. Let $\tau=\sigma|_X$. Then the morphism $\ol\tau$ induced by
	\begin{equation}
		\tau':\swfi{X}\rightarrow\I:\left\{\begin{array}{lclclcl}
			x &\mapsto& x\tau & = & x\sigma \\
			x^{-1} &\mapsto& (x\tau)^{-1} & = & (x\sigma)^{-1} &= &\wh{x}\sigma
		\end{array}\right.
	\end{equation}
	satisfies $u\ol\tau = (u\theta)\ol\sigma = (v\theta)\ol\sigma = v\ol\tau$ for all $(u,v)\in R$, as required.
\end{proof}

We now formalise the process of triangulating a system of equations. Let $R$ be a system of inverse monoid equations with variable set $X$. Observe that any mapping $\sigma:\swfi{X}\rightarrow\I$ is a solution to $(1,1)\in(\swfi{X})^*\times(\swfi{X})^*$, and if $\sigma$ is a solution to $(u,v)\in(\swfi{X})^*\times(\swfi{X})^*$ then $\sigma$ is a solution to $(v,u)$. So we can assume without loss of generality that $(1,v)\notin R$ for any $v\in(\swfi{X})^*$. For each $\eta=(u,v)\in R$ we have $u=u_1\cdots u_j$ and $v=v_1\cdots v_k$ for some $u_i,v_i\in\swfi{X}$. Define
\begin{equation}
	R_\eta = \{(y_1,z_1)\}\cup U_\eta\cup V_\eta \qquad\WHERE\qquad U_\eta = \{(y_i,u_iy_{i+1}):1\le i< j\}\cup\{(y_j,u_j)\}
\end{equation}
and
\begin{equation}
	V_\eta = \left\{\begin{array}{ll}
		\{(z_1,1)\} & \text{if $v=1$} \\
		\{(z_i,v_iz_{i+1}):1\le i< k\}\cup\{(z_k,v_k)\} & \text{otherwise.}
	\end{array}\right.
\end{equation}
Here, $Y_\eta=\{y_1,\dotsc,y_j\}$ and $Z_\eta=\{z_1,\dotsc,z_k\}$ are new sets of variables. Define the following system of equations
\begin{equation}
	\Xi(R) = \bigcup_{\eta\in R} R_\eta \quad\text{with variable set}\quad X\cup W_R \quad\WHERE\quad W_R=\bigcup_{\eta\in R}\left(Y_\eta\cup Z_\eta\right).
\end{equation}
Similar to \cref{lem:rem_inv_ele_in_eqns}, the following lemma describes the relationship between solutions for $R$ and $\Xi(R)$.

\begin{lemma}\label{lem:split_inv_eqns}
	Let $\I$ be an inverse monoid and $R\subseteq (\swfi{X})^*\times(\swfi{X})^*$. If $\tau:X\rightarrow \I$ is a solution to~$R$ then there exists a solution $\sigma:X\cup W_R\rightarrow\I$ to $\Xi(R)$ such that $x\sigma=x\tau$ for all $x\in X$. Conversely, if $\sigma$ is a solution to~$\Xi(R)$ then~$\sigma|_X$ is a solution to $R$.
\end{lemma}

\begin{proof}
	Suppose $\tau:X\rightarrow\I$ is a solution to $R$. If $w\in W_R$ then $w=y_i\in Y_\eta$ or $w=z_i\in Z_\eta$ for some ${\eta=(u_1\cdots u_j,v_1\cdots v_k)\in R}$ and $1\le i\le j,k$. Define the mapping
	\begin{equation}
		\sigma:X\cup W_R\rightarrow\I:\left\{\begin{array}{ll}
			x\mapsto x\tau & \text{if $x\in X$} \\
			w\mapsto (u_iu_{i+1}\cdots u_j)\ol\tau & \text{if $w=y_i\in Y_\eta$} \\
			w\mapsto (v_iv_{i+1}\cdots v_k)\ol\tau & \text{if $w=z_i\in Z_\eta$.}
		\end{array}\right.
	\end{equation}
	If $\xi\in\Xi(R)$ then $\xi\in R_\eta$ for some $\eta=(u_1\cdots u_j,v_1\cdots v_k)\in R$. There are three cases. \medskip
	\begin{description}
		\item[Case 1] If $\xi=(y_j,u_j)\in U_\eta$ then $y_j\ol\sigma = u_j\ol\tau = u_j\ol\sigma$ since $u_j\in \swfi{X}$. A similar argument shows $\sigma$ is a solution to $\xi=(z_k,v_k)\in V_\eta$ or $\xi=(z_1,1)\in V_\eta$.
		\item[Case 2] If $\xi=(y_i,u_iy_{i+1})\in U_\eta$ for some $1\le i\le j$ then
		\begin{equation}
			y_i\ol\sigma \ = \ (u_iu_{i+1}\cdots u_j)\ol\tau \ = \ (u_i\ol\tau)\big((u_{i+1}\cdots u_j)\ol\tau\big) \ = \ (u_i\ol\sigma)(y_{i+1}\ol\sigma) \ = \ (u_iy_{i+1})\ol\sigma.
		\end{equation}
		A similar argument shows $z_i\ol\sigma=(v_iz_{i+1})\ol\sigma$ if $\xi=(z_i,v_iz_{i+1})\in U_\eta$ for some $1\le i\le k$.
		\item[Case 3] If $\xi=(y_1,z_1)$ then $y_1\ol\sigma=(u_1\dotsc u_j)\ol\tau=u\ol\tau=v\ol\tau=(v_1\dotsc v_k)\ol\tau=z_1\ol\sigma$ since $(u,v)\in R$ and $\tau$ is a solution to $R$.
	\end{description}
	
	Conversely, suppose that $\sigma:X\cup W_R\rightarrow \I$ is a solution to $\Xi(R)$. Let $\ol\tau:(\swfi{X})^*\rightarrow\I$ be the morphism induced by $\tau=\sigma|_X$. If $\eta=(u,v)=(u_1\cdots u_j,v_1\cdots v_k)\in R$ then
	\begin{equation}
		\begin{array}{rclll}
			u\ol\tau & = & (u_1\cdots u_{j-2}u_{j-1}u_j)\ol\tau \\
			& = & (u_1\ol\sigma)\cdots(u_{j-2}\ol\sigma)(u_{j-1}\ol\sigma)(u_j\ol\sigma) \\
			& = & (u_1\ol\sigma)\cdots(u_{j-2}\ol\sigma)(u_{j-1}\ol\sigma)(y_j\ol\sigma) & & \text{since $(u_j,y_j)\in\Xi(R)$}  \\
			& = & (u_1\ol\sigma)\cdots(u_{j-2}\ol\sigma)(y_{j-1}\ol\sigma) & & \text{since $(y_{j-1},u_{j-1}y_j)\in\Xi(R)$}  \\
			& = & \cdots \ \ = \ \ y_1\ol\sigma.
		\end{array}
	\end{equation}
	Similarly, $v\ol\tau = z_1\ol\sigma$ and thus $u\ol\tau=v\ol\tau$ since $(y_1,z_1)\in\Xi(R)$.
\end{proof}

We now describe how to convert the triangulated equations into a PA formula. Let $R$ be a system of inverse monoid equations with variable set $X$. For each $\eta\in\Xi(R)$ define the open formula
\begin{equation}
	\psi(\eta) = \left\{\begin{array}{ll}
		x_\lambda=0 \land x_\rho=0 & \text{if $\eta=(x,1)$} \\[1.5mm]
		x_\lambda=y_\lambda \land x_\rho=y_\rho & \text{if $\eta=(x,y)$} \\[1.5mm]
		\setlength{\arraycolsep}{0pt}
		\begin{array}{l}
			\big(y_\rho\ge z_\lambda \land x_\lambda=y_\lambda \land x_\rho=z_\rho-z_\lambda+y_\rho\big) \\[1pt]
			\quad {}\lor\big(\neg(y_\rho\ge z_\lambda) \land x_\lambda=y_\lambda-y_\rho+z_\lambda \land x_\rho=z_\rho\big)
		\end{array}
		& \text{if $\eta=(x,yz)$,}
	\end{array}\right.
\end{equation}
where $x,y,z\in X\cup W_R$ and $x_\lambda,x_\rho$, etc. are free variables. Define the formula
\begin{equation}
	\Psi(\Xi(R)) = \bigwedge_{v\in V}v\ge0\land \bigwedge_{\eta\in\Xi(R)} \psi(\eta) \quad\text{with variable set}\quad V=\{x_\lambda,x_\rho:x\in X\cup W_R\}.
\end{equation}
The following lemma shows that the solutions to the system $\Xi(\Theta(R))$ over $\B$ are in one-to-one correspondence with satisfying assignments for $\Psi(\Xi(\Theta(R)))$ over $\N$.

\begin{lemma}\label{thm:eqn_sol_iff_formula_satisfied}
	Let $R\subseteq (\swfi{X})^*\times(\swfi{X})^*$. If $\tau:X\rightarrow\B$ is a solution to $R$ then there exists a satisfying assignment $\varsigma:V\rightarrow\N$ for $\Psi(\Xi(\Theta(R)))$ such that $(x_\lambda\varsigma,x_\rho\varsigma) = x\tau$ for all $x\in X$. Conversely, if $\varsigma$ is a satisfying assignment for $\Psi(\Xi(\Theta(R)))$ then the mapping $\tau:X\rightarrow\B$ given by $x\mapsto(x_\lambda\varsigma,x_\rho\varsigma)$ is a solution to $R$ over $\B$.
\end{lemma}
\begin{proof}
	Write $\bm{X}=\big(X\cup \wh{X}\big)\cup W_{\Theta(R)}$. Suppose $\tau$ is a solution to $R$ over $\B$. By \cref{lem:rem_inv_ele_in_eqns,lem:split_inv_eqns} there is a corresponding solution 
	\begin{equation}
		\sigma:\bm{X}\rightarrow \B \quad\text{satisfying}\quad x\sigma=x\tau \quad\text{to $\Xi(\Theta(R))$.}
	\end{equation}
	For brevity, we will write $\ol{x}=x\ol\sigma$ for $x\in \bm{X}$. Define the mapping
	\begin{equation}
		\varsigma:V\rightarrow\N:\left\{\begin{array}{ll}
			x_\lambda\mapsto \lambda(\ol{x}) \\
			x_\rho\mapsto \rho(\ol{x})
		\end{array}\right.
		\quad\FOR\quad x\in \bm{X} \quad\text{and write}\quad \wt{v}=v\varsigma \quad\FOR\quad v\in V
	\end{equation}
	Hence $(\wt{x}_\lambda,\wt{x}_\rho)=\ol{x}$ for all $x\in\bm{X}$. Each $\eta\in\Xi(\Theta(R))$ falls into one of three cases.\medskip
	\begin{description}
		\item[Case 1] If $\eta=(x,1)$ for some $x\in\bm{X}$ then $(\wt{x}_\lambda,\wt{x}_\rho) = \ol{x} = \ol{1} = (0,0)$.
		\item[Case 2] If $\eta=(x,y)$ for some $x,y\in\bm{X}$ then $(\wt{x}_\lambda,\wt{x}_\rho)= \ol{x} = \ol{y} = (\wt{y}_\lambda,\wt{y}_\rho)$.
		\item[Case 3] If $\eta=(x,yz)$ for some $x,y,z\in \bm{X}$ then
		\begin{equation}
			\setlength{\arraycolsep}{3pt}
			\begin{array}{rclll}
				(\wt{x}_\lambda,\wt{x}_\rho) & = & \ol{x} \ = \ \ol{y}\,\ol{z} \\[1mm]
				& = & \Big(\lambda(\ol{y})-\rho(\ol{y})+\max\big(\rho(\ol{y}), \lambda(\ol{z})\big),\ \rho(\ol{z})-\lambda(\ol{z})+\max\big(\rho(\ol{y}),\lambda(\ol{z})\big)\Big) \\[1.5mm]
				& = & \big(\wt{y}_\lambda-\wt{y}_\rho+\max(\wt{y}_\rho, \wt{z}_\lambda),\ \wt{z}_\rho-\wt{z}_\lambda+\max(\wt{y}_\rho,\wt{z}_\lambda)\big) \\[1.5mm]
				& = & \left\{\begin{array}{lll}
					(\wt{y}_\lambda,\ \wt{z}_\rho-\wt{z}_\lambda+\wt{y}_\rho) & & \text{if $\wt{y}_\rho\ge \wt{z}_\lambda$} \\[1mm]
					(\wt{y}_\lambda-\wt{y}_\rho+\wt{z}_\lambda,\ \wt{z}_\rho) & & \text{if $\wt{y}_\rho< \wt{z}_\lambda$.}
				\end{array}\right.
			\end{array}
		\end{equation}
	\end{description}
	In each case, the assignment $\varsigma$ satisfies $\psi(\eta)$. Clearly $v\ge0$ for all $v\in V$ since $\im(\varsigma)=\N$. Therefore $\varsigma$ satisfies $\Psi(\Xi(\Theta(R)))$.
	
	For the converse, suppose $\varsigma$ is a satisfying assignment for $\Psi(\Xi(\Theta(R)))$. Define the mapping
	\begin{equation}
		\sigma:\bm{X}\rightarrow \B:x\mapsto (\wt{x}_\lambda,\wt{x}_\rho) \quad\WHERE\quad \wt{v}=v\varsigma \quad\FOR\quad v\in V. 
	\end{equation}
    Note $\wt{v}\in\N$ for all $v\in V$ since $\varsigma$ is a solution to $\Psi(\Xi(\Theta(R)))$, so we indeed have $\im(\sigma)\subseteq\N\times\N$. As before, we write $\ol{x}=x\ol\sigma$ for $x\in\bm{X}$ and consider three cases for $\eta\in\Xi(\Theta(R))$. \medskip
	\begin{description}
		\item[Case 1] If $\eta=(x,1)$ then the truth of $\psi(\eta)$ implies $\ol{x}=(\wt{x}_\lambda,\wt{x}_\rho)=(0,0)=\ol{1}$.
		\item[Case 2] If $\eta=(x,y)$ then $\ol{x}=(\wt{x}_\lambda,\wt{x}_\rho)=(\wt{y}_\lambda,\wt{y}_\rho)=\ol{y}$.
		\item[Case 3] If $\eta=(x,yz)$ then $\ol{x} = \big(\wt{y}_\lambda-\wt{y}_\rho+\max(\wt{y}_\rho,\wt{z}_\lambda), \wt{z}_\rho-\wt{z}_\lambda+\max(\wt{y}_\rho,\wt{z}_\lambda)\big) = \ol{y}\,\ol{z}$.
	\end{description} \medskip
	Therefore $\sigma$ is a solution for $\Xi(\Theta(R))$ over $\B$. By \cref{lem:split_inv_eqns}, $\sigma|_{X\cup \wh{X}}$ is a solution to $\Theta(R)$, which by \cref{lem:rem_inv_ele_in_eqns} implies $\tau=\big(\sigma|_{X\cup \wh{X}}\big)\big|_X=\sigma|_{X}$ is a solution to $R$.
\end{proof}

Now that we have shown how to reduce a system of equations over $\B$ to a PA formula, all that remains is to encode, as a \EPAGRP formula, the necessary and sufficient conditions for the image of a morphism $\Inv\pres{X}{R}\rightarrow\B$ to equal $\B$. Given an alphabet $X$, define the formula
\begin{equation}
	\Gamma(X) = \bigvee_{x\in X} \big((x_\lambda=0\land x_\rho\ge1)\lor (x_\lambda\ge1\land x_\rho=0)\big) \quad\text{over}\quad \{x_\lambda,x_\rho:x\in X\}.
\end{equation}
Given a presentation $\Inv\pres{X}{R}$, define the \EPAGRP formula
\begin{equation}
	\Phi(X,R) = \Psi(\Xi(R))\land\Gamma(X)\land\gcd\{x_\lambda-x_\rho:x\in X\}=1
\end{equation}
over the variable set $V=\{x_\lambda,x_\rho: x\in\big(X\cup \wh{X}\big)\cup W_{\Theta(R)}\}$.

\begin{lemma}\label{thm:epi_iff_PA-GCD-formula}
	Let $\I=\Inv\pres{X}{R}$. There exists an epimorphism $\I\rightarrow\B$ if and only if there exists a satisfying assignment for the formula $\Phi(X,R)$.
\end{lemma}
\begin{proof}
	Suppose there exists an epimorphism $\I\rightarrow\B$. By \cref{thm:von-dyck}(ii), this implies that there is a solution $\tau:X\rightarrow\B$ such that $\Inv\langle X\tau\rangle=\B$. By \cref{thm:eqn_sol_iff_formula_satisfied}, there exists a satisfying assignment $\varsigma:V\rightarrow\N$ for $\Psi(\Xi(R))$ such that $(x_\lambda\varsigma,x_\rho\varsigma)=x\tau$. By \cref{thm:gen_set_B}(i), $\Inv\langle X\tau\rangle=\B$ implies that there exists $x\in X$ such that $x\tau=(0,m)$ or $(m,0)$ for some $m>0$. Hence this $x_\lambda\varsigma=\lambda(x\tau)$ and $x_\rho\varsigma=\rho(x\tau)$ satisfy $\Omega(X)$. Finally, \cref{thm:gen_set_B}(ii), implies that 
	\begin{equation}
		\gcd\{x_\lambda\varsigma-x_\rho\varsigma:x\in X\} = \gcd\{\lambda(x\tau)-\rho(x\tau):x\in X\} = \gcd\{\delta(x\tau):x\in X\} = 1.
	\end{equation}
	
	For the converse, suppose that there exists a satisfying assignment $\varsigma:V\rightarrow\N$ to $\Phi(X,R)$. Then, according to \cref{thm:eqn_sol_iff_formula_satisfied}, $\varsigma$ satisfying $\Psi(\Xi(R))$ implies the mapping $\tau:X\rightarrow\B$ given by $x\mapsto(x_\lambda\varsigma,x_\rho\varsigma)$ is a solution to $R$ over $\B$. Since $\varsigma$ satisfies $\Omega(X)$ there exists $x\in X$ such that $x\tau=(x_\lambda\varsigma,x_\rho\varsigma)=(0,m)$ or $(m,0)$ for some $m>0$. Together with
	\begin{equation}
		\gcd\{x_\lambda\varsigma-x_\rho\varsigma:x\in X\} = 1 \quad\implies\quad \gcd\{\delta(x\tau):x\in X\} = 1
	\end{equation}
	we have by \cref{thm:gen_set_B} that $X\tau$ generates $\B$. Therefore there exists an epimorphism $\I\rightarrow\B$ by \cref{thm:von-dyck}(i).
\end{proof}

\begin{proof}[Proof of \cref{thmMAINA:epi_is_NP-complete}]
We have shown that 
\Epi is poly time reducible to \EPAGRP. This combined with \cref{thm:epi_is_NP-hard} gives the result.
\end{proof}

\bibliographystyle{plainurl}
\bibliography{refs}

\end{document}